\documentclass[11pt,a4paper,fleqn]{article}

\usepackage{amssymb,amsmath,amsfonts,amsthm}  %% 
\usepackage[top=30truemm,bottom=30truemm,left=20truemm,right=20truemm]{geometry}
\usepackage{enumerate}  %% 
\usepackage{graphicx}  %% 
\usepackage{epstopdf}
\usepackage{comment}
\usepackage{array}

\usepackage{mathrsfs} %%  
\usepackage{latexsym}
\usepackage{epic,eepic} 

\usepackage[english]{babel}

\usepackage{here}
\theoremstyle{definition} 
\newtheorem{theorem}{Theorem}[section]

\newtheorem{lemma}[theorem]{Lemma}

\newtheorem{definition}[theorem]{Definition}
\newtheorem{remark}[theorem]{Remark}

\numberwithin{equation}{section}
\numberwithin{table}{section}
\numberwithin{figure}{section}

\allowdisplaybreaks[1]
\title{\Large Curvature dimension inequalities on directed graphs}
\author{Taiki Yamada}
\begin{document}
 \maketitle
 
 \begin{abstract}
In this paper, we define the curvature dimension inequalities $CD(m, K)$ on finite directed graphs modifying the case of undirected graphs. As a main result, we evaluate $m$ and $K$ on finite directed graphs. %Moreover, we calculate it on complete graph $K_{2m+1}$ and consider the relationship between $K$ and the Ricci curvature.
 \end{abstract}
 \section{Introduction}
　In Riemannian geometry, we usually use the curvature dimension inequalities $CD(m, K)$ to consider analytic properties of Ricci curvature. The curvature dimension inequalities are defined by the $\Gamma$-calculation. To define $\Gamma$-calculation on graphs, we need the symmetric graph laplacian. In general, the graph laplacian is defined using an adjacent matrix. Jost-Liu \cite{Jo2} defined the curvature dimension inequalities on finite undirected graphs, and found relations between Ricci curvature and $CD(m, K)$ . However, an adjacent matrix of directed graphs is usually asymmetry, so the graph laplacian is also asymmetry. 
 On the other hand, there is an another definition of the graph laplacian by Chung \cite{Chung}. He defined the graph laplacian by using the transition probability matrix, and symmetrized it.\\
 　We have defined Ricci curvature on directed graphs \cite{Yamada}. So, we define the curvature dimension inequalities on finite directed graphs, and find specific $m$ and $K$ in this paper.
 \section{Curvature dimension inequalities on directed graphs}
  Throughout the paper, let $G=(V, E)$ be a directed graph. For $x, y \in V$, we write $(x, y)$ as an edge from $x$ to $y$ if any.  We denote the set of vertices of $G$ by $V(G)$ and the set of edges by $E(G)$.
  \begin{definition}
  (1) A {\em path} from vertex $x$ to vertex $y$ is a sequence of edges\\ $(x, a_{1}), (a_{1}, a_{2}), \cdots ,(a_{n-2}, a_{n-1}), (a_{n-1}, y)$, where $n \geq 1$. We call $n$ the {\em length} of the path.\\
  (2) The {\em distance} $d(x, y)$ between two vertices $x, y \in V$ is the length of a shortest path connecting them.
  \end{definition}
  \begin{remark}
  The distance function has positivity and triangle inequality, but symmetry is not necessarily satisfied.
  \end{remark}
    \begin{definition}
  (1) For any $x \in V$, the {\em out-neighborhood} ({\em in-neighborhood}) of $x$ is defined as
   \begin{eqnarray*}
    S^{\mathrm{out}}(x) := \left\{y \in V \mid (x, y) \in E \right\}\ (S^{\mathrm{in}}(x) := \left\{z \in V \mid (z, x) \in E \right\}).
   \end{eqnarray*}
%   where $S^{\mathrm{in}}(x) := \left\{z \in V \mid (z, x) \in E \right\}$ and 
%   $S^{\mathrm{out}}(x) := \left\{y \in V \mid (x, y) \in E \right\}$.\\
  (2) For all $x \in V$, 
   The {\em degree} of $x$, denoted by $d_{x}$, is the number of edges from $x$. i.e., $d_{x} = |S^{\mathrm{out}}(x)|$.
    \end{definition}
    We sometimes use the following notation.
  \begin{eqnarray*}
	N^{\mathrm{out}}(x) := \left\{ x \right\} \cup S^{\mathrm{out}}(x)\
	(N^{\mathrm{in}}(x) := \left\{ x \right\} \cup S^{\mathrm{in}}(x)).
  \end{eqnarray*}  
  \begin{definition}
  A {\em probability matrix} $M_{\alpha}$ is defined by the following.
\begin{eqnarray*}
(M_{\alpha})_{i, j} = m^{\alpha}_{v_{i}}(v_{j}) :=\begin{cases}
 \displaystyle \alpha,& v_{j} = v_{i}, \\
 \displaystyle \frac{1 - \alpha}{d_{v_{i}}},& (v_{i}, v_{j}) \in E(G), \\
 0,& \mathrm{\mathrm{otherwise}},
\end{cases}
\end{eqnarray*}
where $\alpha \in [0, 1)$ and $(M_{\alpha})_{i, j}$ is an $(i, j)$-element of $M_{\alpha}$. 
 \end{definition}
 \begin{remark}
If $G$ is strongly connected, the matrix $M_{\alpha}$ is irreducible, so by Perron-Frobenius' theorem, there exists  $\phi$ such that all elements of $\phi$ is positive and $\phi M_{\alpha} = \phi$ (see \cite{Chung}).
 \end{remark}
  \begin{definition}
  \label{laplacian}
  For all $f: V(G) \to \mathbb{R}$, a {\em laplacian on directed graphs} is defined by the following.
  \begin{eqnarray*}
   \Delta(f)(v_{i})=\cfrac{1}{2 \phi(v_{i})}\left\{\sum_{v_{j} \in S^{\mathrm{out}}(v_{i})}w_{ij}(f(v_{j}) - f(v_{i})) + \sum_{v_{k} \in S^{\mathrm{in}}(v_{i})}w_{ki}(f(v_{k}) - f(v_{i})) \right\},
  \end{eqnarray*}
  where $w_{ij}=\phi(v_{i})(M_{\alpha})_{i,j}$
  \end{definition}
    \begin{definition}
     \label{gamma}
  For all functions $f$, $g: V(G) \to \mathbb{R}$, a $\Gamma$-$calculation$ is defined by the following.
  \begin{eqnarray*}
   &\Gamma(f, g)& = \cfrac{1}{2}\left\{ \Delta(fg) - f \Delta(g) - g \Delta(f) \right\},\\
   &\Gamma_{2}(f, g)& = \cfrac{1}{2}\left\{ \Delta(\Gamma(f, g)) - \Gamma(f, \Delta(g)) - \Gamma(\Delta(f), g) \right\}.
  \end{eqnarray*}
  \end{definition}
  \begin{definition}
  We say a graph $G$ is satisfies the {\em curvature dimension inequalities} $CD(m, K)$ if for all functions $f: V(G) \to \mathbb{R}$ and for all vertices $v_{i} \in V(G)$,
  \begin{eqnarray}
   \Gamma_{2}(f, f)(v_{i}) \geq \cfrac{1}{m}(\Delta f(v_{i}))^{2} + K(v_{i})\Gamma(f, f)(v_{i}), \nonumber
  \end{eqnarray}
 where $m \in [1, + \infty]$．
  \end{definition}
 \begin{remark}
  We say $K(v_{i})$ is {\em Bakry-Emery's curvature}. In Riemannian geometry, this curvature is a lower bound of the Ricci curvature.
 \end{remark}

 \section{Main theorem}
   \begin{theorem}
    \label{main}
  Let $G$ be a finite directed graph satisfying strongly connected. Then $G$ satisfies $CD(2, C-(1-\alpha))$. i.e. for all vertices $v_{i} \in V(G)$
  \begin{eqnarray}
   \Gamma_{2}(f, f)(v_{i}) \geq \cfrac{1}{2}(\Delta f(v_{i}))^{2} + \left\{C(v_{i}) - (1- \alpha) \right\}\Gamma(f, f)(v_{i}), \nonumber
  \end{eqnarray}
  where $1 > C(v_{i}) = \min_{v_{j} \in S^{\mathrm{out}}(v_{i}), v_{k} \in S^{\mathrm{in}}(v_{i})}\left\{w_{ij}/\phi(v_{j}), w_{ki}/\phi(v_{k})\right\} > 0$.
  \end{theorem}
 We prove some lemmas about the property of the $\Gamma$-calculation and the laplacian before the main theorem.
 \begin{lemma}
 \label{lem1}
  For all functions $f: V(G) \to \mathbb{R}$ and all vertices $v_{i} \in V$, then we have
   \begin{eqnarray*}
    \Gamma(f, f)(v_{i}) = \cfrac{1}{4 \phi(v_{i})}\left\{\sum_{v_{j} \in S^{\mathrm{out}}(v_{i})}w_{ij}(f(v_{j}) - f(v_{i}))^{2} + \sum_{v_{k} \in S^{\mathrm{in}}(v_{i})}w_{ki}(f(v_{k}) - f(v_{i}))^{2} \right\}.
   \end{eqnarray*}
 \end{lemma}
 \begin{proof}
  By the definition \ref{gamma}, $\Gamma(f, f) = \cfrac{1}{2}\left\{ \Delta(f^{2}) - 2f \Delta(f) \right\}$. So, 
  \begin{eqnarray*}
   2 \Gamma(f, f)(v_{i}) &=& \cfrac{1}{2 \phi(v_{i})}\left\{\sum_{v_{j} \in S^{\mathrm{out}}(v_{i})}w_{ij}(f^{2}(v_{j}) - f^{2}(v_{i})) + \sum_{v_{k} \in S^{\mathrm{in}}(v_{i})}w_{ki}(f^{2}(v_{k}) - f^{2}(v_{i})) \right\} \\
  & &- \frac{ f(v_{i})}{ \phi(v_{i})}\left\{\sum_{v_{j} \in S^{\mathrm{out}}(v_{i})}w_{ij}(f(v_{j}) - f(v_{i})) + \sum_{v_{k} \in S^{\mathrm{in}}(v_{i})}w_{ki}(f(v_{k}) - f(v_{i})) \right\}\\
  &=& \cfrac{1}{2 \phi(v_{i})}\left\{\sum_{v_{j} \in S^{\mathrm{out}}(v_{i})}w_{ij}(f(v_{j}) - f(v_{i}))^{2} + \sum_{v_{k} \in S^{\mathrm{in}}(v_{i})}w_{ki}(f(v_{k}) - f(v_{i}))^{2} \right\}.
  \end{eqnarray*}
 \end{proof}
 \begin{lemma}
 \label{lem2}
 For all functions $f: V(G) \to \mathbb{R}$ and all vertices $v_{i} \in V$, then we have
 \begin{eqnarray*}
  \Delta \Gamma(f, f)(v_{i}) =\cfrac{1}{8 \phi(v_{i})}\left\{ \sum_{v_{j} \in N^{\mathrm{out}}(v_{i})}\cfrac{w_{ij}}{\phi(v_{j})} (W^{\mathrm{in}}(v_{j}) + W^{\mathrm{out}}(v_{j}))+ \sum_{v_{k} \in N^{\mathrm{in}}(v_{i})}\cfrac{w_{ki}}{\phi(v_{k})} (W^{\mathrm{in}}(v_{k}) + W^{\mathrm{out}}(v_{k})) \right\},
   \end{eqnarray*}
   where
   \begin{eqnarray*}
   W^{\mathrm{out}}(v_{j}) &=& \sum_{v_{a} \in N^{\mathrm{out}}(v_{j})}w_{ja}((f(v_{a}) - f(v_{j}))^{2} - (f(v_{j}) - f(v_{i}))^{2}\\
   &=& \sum_{v_{a} \in N^{\mathrm{out}}(v_{j})}w_{ja}\left\{(f(v_{a}) - 2 f(v_{j}) + f(v_{i}))^{2} + 2 (f(v_{j}-f(v_{i}))(f(v_{a}) - 2f(v_{j}) + f(v_{i}))\right\},\\
   W^{\mathrm{in}}(v_{j}) &=&\sum_{v_{b} \in N^{\mathrm{in}}(v_{j})}w_{bj}((f(v_{b}) - f(v_{j}))^{2} - (f(v_{j}) - f(v_{i}))^{2}\\
   &=& \sum_{v_{b} \in N^{\mathrm{in}}(v_{j})}w_{bj}\left\{(f(v_{b}) - 2 f(v_{j}) + f(v_{i}))^{2} + 2 (f(v_{j}-f(v_{i}))(f(v_{b}) - 2f(v_{j}) + f(v_{i}))\right\}. 
   \end{eqnarray*}
 \end{lemma}
 \begin{proof}
  By the definition \ref{laplacian} and lemma \ref{lem1},  
   \begin{eqnarray}
   2 \phi (v_{i}) \Delta \Gamma(f, f)(v_{i}) &=& \sum_{v_{j} \in S^{\mathrm{out}}(v_{i})}w_{ij}(\Gamma (f, f)(v_{j}) - \Gamma (f, f)(v_{i}))\\
   &+& \sum_{v_{k} \in S^{\mathrm{in}}(v_{i})}w_{ki}(\Gamma (f, f)(v_{k}) - \Gamma (f, f)(v_{i})).
  \end{eqnarray}
  So, we calculate (3.1). Then we obtain
   \begin{eqnarray}
    &\ & \sum_{v_{j} \in S^{\mathrm{out}}(v_{i})}w_{ij}(\Gamma (f, f)(v_{j}) - \Gamma (f, f)(v_{i})) \nonumber \\
    &=& \sum_{v_{j} \in N^{\mathrm{out}}(v_{i})}w_{ij}\Gamma (f, f)(v_{j}) - \phi(v_{i}) \Gamma (f, f)(v_{i})\nonumber \\
    &=& \sum_{v_{j} \in N^{\mathrm{out}}(v_{i})}\cfrac{w_{ij}}{4 \phi(v_{j})}\left\{\sum_{v_{a} \in N^{\mathrm{out}}(v_{j})}w_{ja}(f(v_{a}) - f(v_{j}))^{2} + \sum_{v_{b} \in S^{\mathrm{in}}(v_{j})}w_{bj}(f(v_{b}) - f(v_{j}))^{2} \right\} \nonumber \\
    &\ & - \cfrac{1}{4}\left\{ \sum_{v_{j} \in N^{\mathrm{out}}(v_{i})}w_{ij}(f(v_{j}) - f(v_{i}))^{2} + \sum_{v_{k} \in S^{\mathrm{in}}(v_{i})}w_{ki}(f(v_{k}) - f(v_{i}))^{2}  \right\} \nonumber \\
    &=& \sum_{v_{j} \in N^{\mathrm{out}}(v_{i})}\cfrac{w_{ij}}{4 \phi(v_{j})} W^{\mathrm{out}}(v_{j})\nonumber \\
    &\ & + \sum_{v_{j} \in N^{\mathrm{out}}(v_{i})}\cfrac{w_{ij}}{4 \phi(v_{j})} \sum_{v_{b} \in S^{\mathrm{in}}(v_{j})}w_{bj}(f(v_{b}) - f(v_{j}))^{2} \\
    &\ & - \cfrac{1}{4}\sum_{v_{k} \in S^{\mathrm{in}}(v_{i})}w_{ki}(f(v_{k}) - f(v_{i}))^{2}.
  \end{eqnarray}
   We calculate (3.2) similarly. Then we obtain
  \begin{eqnarray}
   &\ & \sum_{v_{k} \in S^{\mathrm{in}}(v_{i})}w_{ki}(\Gamma (f, f)(v_{k}) - \Gamma (f, f)(v_{i})) \nonumber \\
    &=& \sum_{v_{k} \in N^{\mathrm{in}}(v_{i})}\cfrac{w_{ki}}{4 \phi(v_{k})} W^{\mathrm{out}}(v_{k})\nonumber \\
    &\ & + \sum_{v_{k} \in N^{\mathrm{in}}(v_{i})}\cfrac{w_{ki}}{4 \phi(v_{k})} \sum_{v_{d} \in S^{\mathrm{in}}(v_{k})}w_{dk}(f(v_{d}) - f(v_{k}))^{2} \\
    &\ & - \cfrac{1}{4}\sum_{v_{j} \in S^{\mathrm{in}}(v_{i})}w_{ij}(f(v_{j}) - f(v_{i}))^{2}
  \end{eqnarray}
   We combine (3.3) and (3.6). Thus,
  \begin{eqnarray*}
   &\ & \sum_{v_{j} \in N^{\mathrm{out}}(v_{i})}\cfrac{w_{ij}}{4 \phi(v_{j})} \sum_{v_{b} \in S^{\mathrm{in}}(v_{j})}w_{bj}(f(v_{b}) - f(v_{j}))^{2} - \cfrac{1}{4}\sum_{v_{j} \in S^{\mathrm{in}}(v_{i})}w_{ij}(f(v_{j}) - f(v_{i}))^{2}\\
   &=& \sum_{v_{j} \in N^{\mathrm{out}}(v_{i})}\cfrac{w_{ij}}{4 \phi(v_{j})} W^{\mathrm{in}}(v_{j})
  \end{eqnarray*}
  We combine (3.4) and (3.5) similarly, and we obtain $\sum_{v_{k} \in N^{\mathrm{in}}(v_{i})}\cfrac{w_{ki}}{4 \phi(v_{k})} W^{\mathrm{in}}(v_{k})$. 
   \end{proof}
% \begin{remark}
%  We can transform the $(f(v_{a}) - f(v_{j}))^{2} - (f(v_{j}) - f(v_{i}))^{2}$ in $W^{\mathrm{out}}(v_{j})$ into the following.
%    \begin{eqnarray*}
%      &\ & (f(v_{a}) - f(v_{j}))^{2} - (f(v_{j}) - f(v_{i}))^{2}\\
%      &=& (f(v_{a}) - 2 f(v_{j}) + f(v_{i}))^{2} + 2 (f(v_{j}-f(v_{i}))(f(v_{a}) - 2f(v_{j}) + f(v_{i}))
%    \end{eqnarray*}
% \end{remark}
  \begin{lemma}
   \label{lem3}
  For all functions $f: V(G) \to \mathbb{R}$ and all vertices $v_{i} \in V$, then we have
   \begin{eqnarray*}
   2 \Gamma(\Delta f, f)(v_{i})  &=& (\Delta f)^{2}(v_{i}) +\cfrac{1}{2 \phi(v_{i})}\left(  \sum_{v_{j} \in S^{\mathrm{out}}(v_{i})}w_{ij}W_{\Delta}(v_{j})+  \sum_{v_{k} \in S^{\mathrm{in}}(v_{i})}w_{ki}W_{\Delta}(v_{k})\right),
    \end{eqnarray*}
 where $W_{\Delta}(v_{j}) = \Delta f(v_{j})(f(v_{j}) - f(v_{i}))$.
  \end{lemma}
  \begin{proof}
  By the definition of \ref{gamma}, $2 \Gamma(\Delta f, f)= \Delta(\Delta f, f) - (\Delta f)^{2} - f \Delta(\Delta f) $. So, 
  \begin{eqnarray*}
  2 \phi(v_{i}) \Delta(\Delta f, f)(v_{i}) &=&  \sum_{ v_{j} \in S^{\mathrm{out}}(v_{i})}w_{ij}(f(v_{j})\Delta f(v_{j})- f(v_{i})\Delta f(v_{i}))\\
   &+&  \sum_{ v_{k} \in S^{\mathrm{in}}(v_{i})}w_{ki}(f(v_{k}) \Delta f(v_{k})- f(v_{i})\Delta f(v_{i})) \\
   &=& \sum_{ v_{j} \in S^{\mathrm{out}}(v_{i})}w_{ij}(\Delta f(v_{j})- \Delta f(v_{i}))f(v_{j})
    + \sum_{ v_{j} \in S^{\mathrm{out}}(v_{i})}w_{ij}(f(v_{j})- f(v_{i})) \Delta f(v_{i})\\
   &+& \sum_{ v_{k} \in S^{\mathrm{in}}(v_{i})}w_{ki}(\Delta f(v_{k})- \Delta f(v_{i}))f(v_{k})
    + \sum_{ v_{k} \in S^{\mathrm{in}}(v_{i})}w_{ki}(f(v_{k})- f(v_{i})) \Delta f(v_{i})\\ 
   &=& \sum_{ v_{j} \in S^{\mathrm{out}}(v_{i})}w_{ij}(\Delta f(v_{j})- \Delta f(v_{i}))(f(v_{j}) - f(v_{i}))\\
   &+& \sum_{ v_{k} \in S^{\mathrm{in}}(v_{i})}w_{ki}(\Delta f(v_{k})- \Delta f(v_{i}))(f(v_{k}) -f(v_{i}))\\
   &+& f(v_{i}) \left\{ \sum_{ v_{j} \in S^{\mathrm{out}}(v_{i})}w_{ij}(\Delta f(v_{j})- \Delta f(v_{i}))+\sum_{ v_{k} \in S^{\mathrm{in}}(v_{i})}w_{ki}(\Delta f(v_{k})- \Delta f(v_{i})) \right\}\\
   &+& \Delta f(v_{i}) \left\{ \sum_{ v_{j} \in S^{\mathrm{out}}(v_{i})}w_{ij}(f(v_{j})- f(v_{i})) + \sum_{ v_{k} \in S^{\mathrm{in}}(v_{i})}w_{ki}(f(v_{k})- f(v_{i})) \right\}\\
   &=& \sum_{ v_{j} \in S^{\mathrm{out}}(v_{i})}w_{ij}\Delta f(v_{j})(f(v_{j}) - f(v_{i})) \sum_{ v_{k} \in S^{\mathrm{in}}(v_{i})}w_{ki}\Delta f(v_{k})(f(v_{k}) -f(v_{i}))\\
%   &+& \Delta f(v_{i}) \left\{ \sum_{ v_{j} \in S^{\mathrm{out}}(v_{i})}w_{ij}(f(v_{j})- f(v_{i})) + \sum_{ v_{k} \in S^{\mathrm{in}}(v_{i})}w_{ki}(f(v_{k})- f(v_{i})) \right\}\\
   &+& 2 \phi (v_{i}) \left\{ 2 (\Delta f)^{2}(v_{i}) + f(v_{i})\Delta(\Delta f)(v_{i}) \right\}.
  \end{eqnarray*}
  The proof is completed.
  \end{proof}
  We use these lemmas and prove the theorem \ref{main}.
  \begin{proof}[Proof of theorem \ref{main}]
  By the definition of \ref{gamma}, $2 \Gamma_{2}(f, f)= \Delta \Gamma( f, f) - 2 \Gamma (\Delta f, f) $. By the lemma \ref{lem3}, 
   \begin{eqnarray*}
   2 \phi(v_{i})(2 \Gamma(\Delta f, f)(v_{i})- (\Delta f)^{2}(v_{i}) )  &=& \cfrac{1}{2 \phi(v_{i})}\left(  \sum_{v_{j} \in S^{\mathrm{out}}(v_{i})}w_{ij}W_{\Delta}(v_{j})+  \sum_{v_{k} \in S^{\mathrm{in}}(v_{i})}w_{ki}W_{\Delta}(v_{k})\right)\\
   &=&  \sum_{v_{j} \in S^{\mathrm{out}}(v_{i})} \cfrac{w_{ij}}{2 \phi (v_{j})}\sum_{v_{a} \in S^{\mathrm{out}}(v_{j})}w_{ja}(f(v_{a}) - f(v_{j}))(f(v_{j}) -f(v_{i}))\\
   &+&  \sum_{v_{j} \in S^{\mathrm{out}}(v_{i})} \cfrac{w_{ij}}{2 \phi (v_{j})}\sum_{v_{b} \in S^{\mathrm{in}}(v_{j})}w_{bj}(f(v_{b}) - f(v_{j}))(f(v_{j}) -f(v_{i}))\\
   &+&  \sum_{v_{k} \in S^{\mathrm{out}}(v_{i})} \cfrac{w_{ki}}{2 \phi (v_{k})}\sum_{v_{c} \in S^{\mathrm{out}}(v_{k})}w_{kc}(f(v_{c}) - f(v_{k}))(f(v_{k}) -f(v_{i}))\\
   &+&  \sum_{v_{k} \in S^{\mathrm{out}}(v_{i})} \cfrac{w_{ki}}{2 \phi (v_{k})}\sum_{v_{d} \in S^{\mathrm{in}}(v_{k})}w_{dk}(f(v_{d}) - f(v_{k}))(f(v_{k}) -f(v_{i})).
    \end{eqnarray*}
    On the other hand, $W^{\mathrm{out}}(v_{j})$ in the lemma \ref{lem2} can be transformed into the following.
    \begin{eqnarray*}
       W^{\mathrm{out}}(v_{j}) &=& \sum_{v_{a} \in N^{\mathrm{out}}(v_{j})}\left\{(f(v_{a}) - 2 f(v_{j}) + f(v_{i}))^{2} + 2 (f(v_{j}-f(v_{i}))(f(v_{a}) - 2f(v_{j}) + f(v_{i}))\right\}\\
       &=& \sum_{v_{a} \in N^{\mathrm{out}}(v_{j})}\left\{(f(v_{a}) - 2 f(v_{j}) + f(v_{i}))^{2} + 2 (f(v_{j})-f(v_{i}))(f(v_{a}) - f(v_{j})) - 2(f(v_{j}) - f(v_{i}))^{2} \right\}.
   \end{eqnarray*}
         $ W^{\mathrm{in}}(v_{j})$, $W^{\mathrm{out}}(v_{k})$, and $W^{\mathrm{in}}(v_{k})$ can be transformed similarly. So, the second terms in $W^{\mathrm{in}}$ and $W^{\mathrm{out}}$ are set off  the value of $2 \phi(v_{i})(2 \Gamma(\Delta f, f)(v_{i})- (\Delta f)^{2}(v_{i}) )$. Moreover, the third terms in $W^{\mathrm{in}}$ and $W^{\mathrm{out}}$ combine, then we obtain $-2 \Gamma(f, f)(v_{i})$. Thus, we have
    \begin{eqnarray*}
     \Gamma_{2}(f, f)(v_{i}) = \cfrac{1}{2}(\Delta f)^{2}(v_{i}) - \Gamma(f, f)(v_{i}) + H(f)(v_{i}),
    \end{eqnarray*}
    where 
    \begin{eqnarray*}
     16 \phi (v_{i}) H(f)(v_{i}) &:=&\sum_{v_{j} \in N^{\mathrm{out}}(v_{i})}\cfrac{w_{ij}}{\phi(v_{j})}\sum_{v_{a} \in N^{\mathrm{out}}(v_{j})}w_{ja}(f(v_{a}) - 2 f(v_{j}) + f(v_{i}))^{2}\\
     &+& \sum_{v_{j} \in N^{\mathrm{out}}(v_{i})}\cfrac{w_{ij}}{\phi(v_{j})}\sum_{v_{b} \in N^{\mathrm{in}}(v_{j})}w_{bj}(f(v_{b}) - 2 f(v_{j}) + f(v_{i}))^{2}\\
     &+&\sum_{v_{k} \in N^{\mathrm{in}}(v_{i})}\cfrac{w_{ki}}{\phi(v_{k})}\sum_{v_{c} \in N^{\mathrm{out}}(v_{k})}w_{kc}(f(v_{c}) - 2 f(v_{k}) + f(v_{i}))^{2}\\
     &+& \sum_{v_{k} \in N^{\mathrm{in}}(v_{i})}\cfrac{w_{ki}}{\phi(v_{k})}\sum_{v_{d} \in N^{\mathrm{in}}(v_{k})}w_{dk}(f(v_{d}) - 2 f(v_{k}) + f(v_{i}))^{2}.
    \end{eqnarray*}
    As a result, it is sufficient to consider only $H(f)(v_{i})$. We remark that the vertex $v_{i}$ belongs to $N^{\mathrm{in}}(v_{j})$ for all $v_{j}$, and $N^{\mathrm{out}}(v_{k})$ for all $v_{k}$. Then, we obtain
    \begin{eqnarray*}
     16 \phi (v_{i}) H(f)(v_{i}) &=& \sum_{v_{j} \in S^{\mathrm{out}}(v_{i})}\cfrac{w_{ij}}{\phi(v_{j})}\sum_{v_{a} \in S^{\mathrm{out}}(v_{j})}w_{ja}(f(v_{a}) - 2 f(v_{j}) + f(v_{i}))^{2}\\
     &+& \alpha \sum_{v_{j} \in S^{\mathrm{out}}} w_{ij} (f(v_{j}) - f(v_{i}))^{2} + \alpha \sum_{v_{j} \in S^{\mathrm{out}}} w_{ij} (f(v_{j}) - f(v_{i}))^{2}\\
     &+& \sum_{v_{j} \in S^{\mathrm{out}}(v_{i})}\cfrac{w_{ij}}{\phi(v_{j})}\sum_{v_{b} \in S^{\mathrm{in}}(v_{j})}w_{bj}(f(v_{b}) - 2 f(v_{j}) + f(v_{i}))^{2} \\
     &+& \alpha \sum_{v_{k} \in S^{\mathrm{in}}} w_{ki} (f(v_{k}) - f(v_{i}))^{2} + \alpha \sum_{v_{j} \in S^{\mathrm{out}}} w_{ij} (f(v_{j}) - f(v_{i}))^{2}\\
     &+&\sum_{v_{k} \in S^{\mathrm{in}}(v_{i})}\cfrac{w_{ki}}{\phi(v_{k})}\sum_{v_{c} \in S^{\mathrm{out}}(v_{k})}w_{kc}(f(v_{c}) - 2 f(v_{k}) + f(v_{i}))^{2}\\
     &+& \alpha \sum_{v_{j} \in S^{\mathrm{out}}} w_{ij} (f(v_{j}) - f(v_{i}))^{2} + \alpha \sum_{v_{k} \in S^{\mathrm{in}}} w_{ki} (f(v_{k}) - f(v_{i}))^{2}\\
     &+& \sum_{v_{k} \in S^{\mathrm{in}}(v_{i})}\cfrac{w_{ki}}{\phi(v_{k})}\sum_{v_{d} \in S^{\mathrm{in}}(v_{k})}w_{dk}(f(v_{d}) - 2 f(v_{k}) + f(v_{i}))^{2}\\
     &+& \alpha \sum_{v_{k} \in S^{\mathrm{in}}} w_{ki} (f(v_{k}) - f(v_{i}))^{2} + \alpha \sum_{v_{k} \in S^{\mathrm{in}}} w_{ki} (f(v_{k}) - f(v_{i}))^{2}\\
     &=&\sum_{v_{j} \in S^{\mathrm{out}}(v_{i})}\cfrac{w_{ij}}{\phi(v_{j})}\sum_{v_{a} \in S^{\mathrm{out}}(v_{j})}w_{ja}(f(v_{a}) - 2 f(v_{j}) + f(v_{i}))^{2}\\
     &+& \sum_{v_{j} \in S^{\mathrm{out}}(v_{i})}\cfrac{w_{ij}}{\phi(v_{j})}\sum_{v_{b} \in S^{\mathrm{in}}(v_{j})}w_{bj}(f(v_{b}) - 2 f(v_{j}) + f(v_{i}))^{2}\\
     &+& \sum_{v_{k} \in S^{\mathrm{in}}(v_{i})}\cfrac{w_{ki}}{\phi(v_{k})}\sum_{v_{c} \in S^{\mathrm{out}}(v_{k})}w_{kc}(f(v_{c}) - 2 f(v_{k}) + f(v_{i}))^{2}\\
     &+& \sum_{v_{k} \in S^{\mathrm{in}}(v_{i})}\cfrac{w_{ki}}{\phi(v_{k})}\sum_{v_{d} \in S^{\mathrm{in}}(v_{k})}w_{dk}(f(v_{d}) - 2 f(v_{k}) + f(v_{i}))^{2} + 16 \alpha \phi(v_{i}) \Gamma(f, f)(v_{i})\\
     &\geq& \sum_{v_{j} \in S^{\mathrm{out}}(v_{i})}\cfrac{4(w_{ij})^{2}}{\phi(v_{j})}(f(v_{j}) - f(v_{i}))^{2}
     + \sum_{v_{k} \in S^{\mathrm{in}}(v_{i})}\cfrac{4(w_{ki})^{2}}{\phi(v_{k})}(f(v_{k})-f(v_{i}))^{2}\\
     &+& 16 \alpha \phi(v_{i}) \Gamma(f, f)(v_{i})\\
     &\geq& 16\phi(v_{i}) (C(v_{i}) + \alpha) \Gamma(f,f)(v_{i}).
    \end{eqnarray*}
    Then, the proof is completed.
    %We can split $H(f)(v_{i})$ into $j = k = i$, $$
      \end{proof}
% \section{The case of complete graphs} 

　{\footnotesize
　T. Yamada: Mathematical Institute, Tohoku University, Aoba, Sendai 980-8578, Japan}\\
　{\footnotesize
　E-mail address: mathyamada@dc.tohoku.ac.jp}
\end{document}